\documentclass[runningheads]{llncs}
\usepackage{amsrefs,amssymb, amsmath, bbold}

\usepackage[T1]{fontenc}
\usepackage{graphicx}
\usepackage{tikz}
\usetikzlibrary{matrix, shapes.misc}
\usepackage{comment,tikz-cd}

\newtheorem{thm}{Theorem}
\newtheorem{lem}[thm]{Lemma}
\newtheorem{cor}[thm]{Corollary}
\newtheorem{defi}[thm]{Definition}

\newtheorem{nota}[thm]{Notation}

\newtheorem{observation}[thm]{Observation}

\newcommand\be{\begin{equation}}
\newcommand\ee{\end{equation}}

\usepackage{amsmath,amsfonts} 

\usepackage[applemac]{inputenc}

\def\bdefi{\begin{defi}}
\def\edefi{\end{defi}}
\def\bnota{\begin{nota}\rm}
\def\enota{\end{nota}}

\def\ZFC{\textup{\textsf{ZFC}}}

\def\osc{\textup{\textsf{osc}}}

\def\({\textup{(}}
\def\){\textup{)}}

\def\N{{\mathbb  N}}
\def\Q{{\mathbb  Q}}
\def\R{{\mathbb  R}}

\def\SSS{\textup{\textsf{S}}}

\def\di{\rightarrow}

\def\asa{\leftrightarrow}

\def\osc{\textup{\textsf{osc}}}

\def\SSS{\textup{\textsf{S}}}

\def\usco{\textup{\textsf{usco}}}

\def\reg{\textup{\textsf{regulated}}}

\def\reg{\textup{\textsf{reg}}}

\def\eps{\varepsilon}

\usepackage{hyperref}

\begin{document}
\title{On the computational properties of ambivalent sets and functions}
%
%
\author{Dag Normann\inst{1}\and Sam Sanders\inst{2}}
\authorrunning{D.\ Normann and S.\ Sanders}
%

\institute{Department of Mathematics, University of Oslo, Norway \email{dnormann@math.uio.no}\and Department of Philosophy II, RUB Bochum, Germany
\email{sasander@me.com} }
%


\setcounter{secnumdepth}{3} 


\maketitle              
\begin{abstract}
Examples of discontinuous functions already appear in the work of Euler, Abel, Dirichlet, Fourier, and Bolzano.  
A ground-breaking discovery due to Baire was that many discontinuous functions are \emph{well-behaved} in that they 
are the pointwise limit of a sequence of continuous functions; the latter form a class nowadays simply called `Baire~1'.  
We shall study a class strictly between the semi-continuous and Baire~1 functions, called the \emph{ambivalent} fuctions.  
In particular, we investigate the computational properties of the class of ambivalent functions and sets, denoted $\bf \Delta$, working with Kleene's S1-S9 schemes.   
Computational equivalences for various standard operations (supremum, Baire 1 representation, \dots) on $\bf \Delta$ are established, including the \emph{structure functional} $\Omega_{\bf \Delta}$ that decides if a given ambivalent set is non-empty.  
A selector is shown to be computable relative to $\Omega_{\bf \Delta}$ and Kleene's quantifier $\exists^{2}$.

\end{abstract}

\section{Introduction}
\subsection{Aim and motivation}
First of all, \emph{computability theory} is a discipline in the intersection of theoretical computer science and mathematical logic where the fundamental question is:
\begin{center}
\emph{given two objects $X $ and $Y$, is $X$ computable from $Y$ {in principle}?}
\end{center} 
Turing's famous `machine' model is the standard reference (\cite{tur37}) in case $X, Y$ are real numbers.  
Our computational study of third-order (and beyond) objects is based on Kleene's S1-S9 in an equivalent formulation (\cite{dagsamXIII, dagsamXV}) using fixed point operators going back to Platek (\cite{pphd}). 
An introduction is in Section \ref{prelim}.  

\smallskip

Secondly, the authors have already studied the computational properties of regulated (\cite{dagsamXIII}) and semi-continuous functions (\cite{dagsamXV}).  
Now, the Baire~1 functions constitute a well-known super-class of the latter with pleasing conceptual properties.  Indeed, the Baire 1 class encompasses a large swath of all discontinuous functions, as already observed by Baire himself (\cite{beren, beren2}). 
Moreover, a function is Baire 1 if it is the pointwise limit of continuous functions, i.e.\ an elementary description not alien to undergraduates.  
Motivated by these observations, we study the \emph{ambivalent} functions, a class strictly between the semi-continuous and Baire 1 functions.
The definition of the former (see Definition \ref{flungf}) can be viewed as a `hybrid' of the definition of Baire 1 and semi-continuity.     

\smallskip

Thirdly, the aforementioned computational study of regulated and semi-continuous functions revolves around so-called \emph{structure functionals}.  
The latter take as input a set $X$ from a given class $\Gamma$ and decide whether it is empty or not, as in the following specification:
\be\tag{$\Omega_{\Gamma}$}
\Omega_\Gamma(X) = \left\{ \begin{array}{ccc}1 & {\rm if} & X \in \Gamma \wedge X \neq \emptyset \\ 0 & {\rm if} & X \in \Gamma \wedge X = \emptyset \\ {\rm undefined} & &{\rm otherwise}  \end{array} \right..
\ee
Now, central to the study of regulated (resp.\ semi-continuous) functions is $\Omega_{b}$ (resp.\ $\Omega_{C}$), which is $\Omega_{\Gamma}$ where $\Gamma$ is the class of finite (resp.\ closed) sets of reals (\cite{dagsamXIII, dagsamXV}).  
For instance, computing the continuous function from the \emph{Urysohn lemma} is exactly as hard\footnote{We always assume that Kleene quantifier $\exists^{2}$ from Section \ref{kelim} is given.} as $\Omega_{C}$ while computing the monotone functions from the \emph{Jordan decomposition theorem} is exactly has hard as $\Omega_{b}$; we repeat that we interpret `computation' via Kleene's S1-S9. 

\smallskip

Finally, the class $\bf \Delta=\bf F_{\sigma}\cap \bf G_{\delta}$ of ambivalent sets (see Definition \ref{char}) generalises the open and closed sets.
A function is ambivalent if its level sets are ambivalent (see Definition \ref{flungf}).  We shall establish the following computational properties of ambivalent sets and functions.
\begin{itemize}
\item We obtain preliminary results for $\bf \Delta$ that are effective versions of theorems from the literature (Section \ref{zelfl})
\item We obtain a cluster theorem for the structure functional $\Omega_{\bf \Delta}$, i.e.\ a number of computational equivalences involving ambivalent functions (Section \ref{clijsterf}). 
\item We show that $\Omega_{\bf \Delta}$ computes a \emph{selector} for $\bf \Delta$-sets, i.e.\ the Axiom of Choice for the latter is effective (Section \ref{slef}). 
\end{itemize}
All required preliminaries may be found in Section \ref{kelim}.  
We add that the notion of ambivalent set goes back to de la Vall\'ee-Poussin circa 1916 (\cite{pussin}). i.e.\ the ambivalent sets have a certain history.

\subsection{Preliminaries and definitions}\label{kelim}
We briefly introduce Kleene's \emph{higher-order computability theory} in Section~\ref{prelim}.
We also introduce some essential axioms (Section~\ref{lll}) and definitions (Section~\ref{cdef}).  A full introduction may be found in e.g.\ \cite{dagsamX}*{\S2} or \cite{longmann}.

\smallskip

Since Kleene's computability theory borrows heavily from type theory, we use common notations from the latter; for instance, the natural numbers are type $0$ objects, denoted $n^{0}$ or $n\in \N$.  
Similarly, elements of Baire space are type $1$ objects, denoted $f\in \N^{\N}$ or $f^{1}$.  Mappings from Baire space $\N^{\N}$ to $\N$ are denoted $Y:\N^{\N}\di \N$ or $Y^{2}$. 
An overview of such notations can be found in e.g.\ \cite{longmann, dagsamXIII}. 

\subsubsection{Kleene's computability theory}\label{prelim}
Our main results are in computability theory and we make our notion of `computability' precise as follows.  
\begin{enumerate}
\item[(I)] We adopt $\ZFC$, i.e.\ Zermelo-Fraenkel set theory with the Axiom of Choice, as the official metatheory for all results, unless explicitly stated otherwise.
\item[(II)] We adopt Kleene's notion of \emph{higher-order computation} as given by his nine clauses S1-S9 (see \cite{longmann}*{Ch.\ 5} or \cite{kleeneS1S9}) as our official notion of `computable'.  We use the formulation based on fixed point operators from \cites{dagsamXIII, dagsamXV}.  
\end{enumerate}
We refer to \cites{longmann, dagsamXIII} for a thorough overview of higher-order computability theory.
We do mention the distinction between `normal' and `non-normal' functionals  based on the following definition from \cite{longmann}*{\S5.4}. 
We only make use of $\exists^{n}$ for $n=2,3$, as defined in Section \ref{lll}.
\bdefi\label{norma}
For $n\geq 2$, a functional of type $n$ is called \emph{normal} if it computes Kleene's quantifier $\exists^{n}$ following S1-S9, and \emph{non-normal} otherwise.  
\edefi
\noindent
It is a historical fact that higher-order computability theory, based on Kleene's S1-S9, has focused primarily on the world of \emph{normal} functionals; this observation can be found \cite{longmann}*{\S5.4} and can be explained by the (then) relative scarcity of interesting non-normal functionals, like the \emph{fan functional}, originally due to Kreisel (see \cite{dagtait} for historical details) and the Superjump due to Gandy.  
The authors have identified interesting \emph{non-normal} functionals, i.e.\ those that compute the objects claimed to exist by:
\begin{itemize}
\item covering theorems due to Heine-Borel, Vitali, and Lindel\"of (\cites{dagsamV}),
\item the Baire category theorem (\cite{dagsamVII, samcsl23}),
\item local-global principles like \emph{Pincherle's theorem} (\cite{dagsamV}),
\item weak fragments of the Axiom of (countable) Choice (\cite{dagsamIX}),
\item the Jordan decomposition theorem and related results (\cites{dagsamXII, dagsamXIII}),
\item the uncountability of $\R$ (\cites{dagsamX, dagsamXI}).
\end{itemize}
This paper continues the study of non-normal functionals that originate from basic properties of Baire 1 functions and sub-classes. 

\smallskip

Finally, the first example of a non-computable non-normal functional, Kreisel's (aka Tait's) fan functional (see \cite{dagtait}), is rather tame: it is computable in $\exists^{2}$. 
By contrast, the functionals based on the previous list are computable in $\exists^{3}$ but not computable in any $\SSS_{k}^{2}$, where the latter decides $\Pi_{k}^{1}$-formulas (see Section \ref{lll}).

\subsubsection{Some comprehension functionals}\label{lll}
In Turing-style computability theory, computational hardness is measured in terms of where the oracle set fits in the well-known comprehension hierarchy.  
For this reason, we introduce some functionals related to \emph{higher-order comprehension} in this section.
We are mostly dealing with \emph{conventional} comprehension here, i.e.\ only parameters over $\N$ and $\N^{\N}$ are allowed in formula classes like $\Pi_{k}^{1}$ and $\Sigma_{k}^{1}$.  

\smallskip
\noindent
First of all, \emph{Kleene's quantifier $\exists^{2}: \N^{\N}\di \{0,1\} $} is the functional satisfying: 
\be\label{muk}\tag{$\exists^{2}$}
(\forall f^{1})\big[(\exists n^{0})(f(n)=0) \asa \exists^{2}(f)=0    \big]. 
\ee
Clearly, $\exists^{2}$ is discontinuous at $f=11\dots$ in the usual epsilon-delta sense. 
In fact, given a discontinuous function on $\N^{\N}$ or $\R$, \emph{Grilliot's trick} computes $\exists^{2}$ from the former, via a rather low-level term of G\"odel's $T$ (see \cite{kohlenbach2}*{\S3}).
Moreover, $\exists^{2}$ computes \emph{Feferman's $\mu^{2}$} defined for any $f^{1}$ as follows:
\be\label{muk2}\tag{$\mu^{2}$}
\mu(f):=
\begin{cases}
n & \textup{if $n^{0}$ is the least natural number such that $f(n)=0$}\\
0 & \textup{if there are no $m^{0}$ such that $f(m)=0$}
\end{cases}.
\ee
Hilbert and Bernays formalise considerable swaths of mathematics using only $\mu^{2}$ (with that name) in \cite{hillebilly2}*{Supplement IV}.

\smallskip
\noindent
Secondly, the \emph{Suslin functional} $\SSS^{2}:\N^{\N}\di \{0,1\}$ (see \cites{kohlenbach2, avi2}) is the unique functional satisfying the following:
\be\label{muk3}\tag{$\SSS^{2}$}
  (\exists g^{1})(\forall n^{0})(f(\overline{g}n)=0)\asa \SSS(f)=0.
\ee
By definition, the Suslin functional $\SSS^{2}$ can decide whether a $\Sigma_{1}^{1}$-formula as in the left-hand side of \eqref{muk3} is true or false.   
We similarly define the functional $\SSS_{k}^{2}$ which decides the truth or falsity of $\Sigma_{k}^{1}$-formulas, given in their Kleene normal form (see e.g.\ \cite{simpson2}*{IV.1.4}).

\smallskip

\noindent
Thirdly,  \emph{Kleene's quantifier $\exists^{3}: (\N^{\N}\di \N)\di \{0,1\} $} is the functional satisfying: 
\be\label{muk4}\tag{$\exists^{3}$}
(\forall Y^{2})\big[  (\exists f^{1})(Y(f)=0)\asa \exists^{3}(Y)=0  \big].
\ee
Hilbert and Bernays introduce functionals in e.g.\ \cite{hillebilly2}*{Supplement IV, p.\ 479} that readily compute $\exists^{3}$.  

\subsubsection{Some definitions}\label{cdef}
We introduce some definitions needed in the below, mostly stemming from mainstream mathematics.
We note that subsets of $\R$ are given by their characteristic functions (Definition \ref{char}), where the latter are common in measure and probability theory.

\smallskip
\noindent
First of all, we make use of the usual definition of (open) set, where $B(x, r)$ is the open ball with radius $r>0$ centred at $x\in \R$.
Note that `RM' stands for `Reverse Mathematics', as the associated representations are used there. 
\bdefi[Set]\label{char}~
\begin{itemize}
\item Subsets $A$ of $ \R$ are given by their characteristic function $F_{A}:\R\di \{0,1\}$, i.e.\ we write $x\in A$ for $ F_{A}(x)=1$ for all $x\in \R$.
\item We write `$A\subset B$' if we have $F_{A}(x)\leq F_{B}(x)$ for all $x\in \R$. 
\item A subset $O\subset \R$ is \emph{open} in case $x\in O$ implies that there is $k\in \N$ such that $B(x, \frac{1}{2^{k}})\subset O$.
\item A subset $F\subset \R$ is \emph{closed} if the complement $\R\setminus F$ is open. 
\item Any $D\subset \R$ is $\bf{F}_{\sigma}$ if it equals $\cup_{n\in \N}F_{n}$ for a sequence $(F_{n})_{n\in \N}$ of closed sets. 
\item Any $D\subset \R$ is $\bf{G}_{\delta}$ if it equals $\cap_{n\in \N}O_{n}$ for a sequence $(O_{n})_{n\in \N}$ of open sets. 
\item We define the class $\bf \Delta$ as $\bf F_{\sigma}\cap G_{\delta}$.  
A set $X\in \bf \Delta$ is called \emph{ambivalent} \(\cites{omaoma, omaoma2}\), \emph{bivalent} \(\cite{dugudugu}\), or \emph{ambiguous} \(\cite{kura, elekes}\).
\item A subset $O\subset \R$ is \emph{RM-open} if there are sequences $(a_{n})_{n\in \N}, (b_{n})_{n\in \N}$ of reals such that $O=\cup_{n\in \N}(a_{n}, b_{n})$.
\item A subset $C\subset \R$ is \emph{RM-closed} if the complement $\R\setminus C$ is RM-open.  
\item An RM-code or representation for $X\in \bf \Delta$ is a sequence of RM-codes of open and closed sets $(O_{n}, F_{n})_{n\in \N}$ with $X=\cap_{n\in \N} O_{n}=\cup_{n\in \N}F_{n}$.  
\end{itemize}
\edefi
We remark that the class of RM-codes for ${\bf \Delta}$-sets is complete $\Pi^1_1$, as is the canonical class of RM-codes for Baire 1 functions.
\noindent

\smallskip

No computational data/additional representation is assumed in our definition of open set.  
As established in \cites{dagsamXII, dagsamXIII, samBIG}, one readily comes across closed sets in basic real analysis (Fourier series) that come with no additional representation. 

\smallskip

Secondly, the following sets are often crucial in proofs in real analysis. 
\bdefi
The sets $C_{f}$ and $D_{f}$ respectively gather the points where $f:\R\di \R$ is continuous and discontinuous.
\edefi
One problem with $C_{f}, D_{f}$ is that `$x\in C_{f}$' involves quantifiers over $\R$.  
In general, deciding whether a given $\R\di \R$-function is continuous at a given real, is as hard as $\exists^{3}$ from Section \ref{lll}.
For these reasons, the sets $C_{f}, D_{f}$ do exist in general, but are not computable in e.g.\ $\exists^{2}$.  For quasi-continuous and semi-continuous functions, these sets \emph{are} definable in $\exists^{2}$ by \cite{dagsamXIII}*{\S2} or \cite{samBIG2}*{Theorem 2.4}.  

\smallskip

Thirdly, we shall study the following notions, many of which are well-known and hark back to Baire, Darboux, Hankel, and Volterra (\cites{beren,beren2,darb, volaarde2,hankelwoot,hankelijkheid}).  
\bdefi\label{flung} 
For $f:\R\di \R$, we have the following definitions:
\begin{itemize}
\item $f$ is \emph{upper semi-continuous} at $x_{0}\in \R$ if for any $k\in \N$, there is $N\in \N$ such that $(\forall y\in B(x_{0}, \frac{1}{2^{N}}))( f(y)< f(x_{0})+\frac{1}{2^{k}} )$, 
\item $f$ is \emph{lower semi-continuous} at $x_{0}\in \R$ if for any $k\in \N$, there is $N\in \N$ such that $(\forall y\in B(x_{0}, \frac{1}{2^{N}}))( f(y)> f(x_{0})-\frac{1}{2^{k}} )$, 
\item $f$ is \emph{regulated} if for every $x_{0}$ in the domain, the `left' and `right' limit $f(x_{0}-)=\lim_{x\di x_{0}-}f(x)$ and $f(x_{0}+)=\lim_{x\di x_{0}+}f(x)$ exist.  
\item $f$ is \emph{Baire 0} if it is a continuous function. 
\item $f$ is \emph{Baire $n+1$} if it is the pointwise limit of a sequence of Baire $n$ functions.
\end{itemize}
\edefi
As to notations, a common abbreviation is `usco' and `lsco' for the first two items.  
Moreover, if a function has a certain weak continuity property at all reals in $\R$ (or its intended domain), we say that the function has that property.  
We shall generally study $\R\di \R$-functions but note that `usco' and `lsco' are also well-defined for $f:[0,1]\di \overline{\R}$ where $\overline{\R}=\R\cup\{+\infty, -\infty\}$ involves 
two special symbols that satisfy $(\forall x\in \R)(-\infty <_{\R} x <_{\R}+\infty )$ by fiat.     

\smallskip

Next, we list the following equivalent definitions for usco and Baire 1 functions, which motivate the definition of ambivalent function.
\bdefi\label{flungf}
For $f:\R\di \R$, we have the following definitions:
\begin{itemize}
\item $f:\R\di \R$ is \emph{usco} in case $\{x\in \R:f(x)\geq a\}$ is closed for any $a\in \R$,
\item $f:\R\di \R$ is \emph{Baire 1} in case $f^{-1}(V)$ is $\bf{F}_{\sigma}$ for any open $V\subset \R$ \(\cite{myerson,overderooie}\),
\item $f:\R\di \R$ is \emph{ambivalent} in case $\{x\in \R:f(x)<a\}$ and $\{x\in \R:f(x)>a\}$ are $\bf{\Delta}$ for any $a\in \R$ \(\cites{omaoma, omaoma2}\).  
\end{itemize}
\edefi

\section{A cluster theorem}
We study the structure functional for ambivalent sets, called $\Omega_{\bf \Delta}$, and establish a cluster theorem (Section \ref{clijsterf}). 
The required preliminaries are in Section \ref{zelfl}.
\subsection{Preliminaries}\label{zelfl}
We obtain some preliminary results needed for Section \ref{clijsterf}, which are often effective versions of known results from the literature. 

\smallskip

First of all, the following is a partial effective version of \cite{overderooie}*{Lemma 11.6}.  
\begin{lem}\label{trux2}
Let $X\in \bf \Delta$ be given. 
Then $\exists^{2}$ computes a Baire 1 representation for $\mathbb{1}_{X}$ from an RM-code of $X$.
\end{lem}
\begin{proof}
First of all, we recall that the (continuous) distance function $d(x, C)=\inf_{y\in C}d(x, y)$ exists in case $C$ is RM-closed; the usual interval-halving technique goes through using $\exists^{2}$.
Now, suppose $X=\cap_{n\in \N}G_{n}=\cup_{m\in \N}F_{m}$ where the $G_{n}$ (resp.\ $F_{m}$) are 
open (resp.\ closed) sets given via RM codes and where $G_{n+1}\subseteq G_{n}$ and $F_{m}\subseteq F_{m+1} $.  
 Use $\exists^{2}$ to define the continuous function $f_{n}(x):=\frac{d(x, \R\setminus G_{n})}{d(x, F_{n})+d(x, \R\setminus G_{n})}$ from these codes.  
 Now observe that $\lim_{n\di \infty }f_{n}(x)=\mathbb{1}_{X}(x)$ for all $x\in \R$, as required.
 \qed
\end{proof}
Secondly, the following lemma is an effective version of \cite{overderooie}*{Lemma 11.7}.
\begin{lem}\label{trux3}
Let $(F_{n,m})_{n,m\in \N}$ be a double sequence of continuous functions, let $f_{n}$ be $\lim_{m\di \infty}F_{n, m}$, and suppose $(f_{n})_{n\in \N}$ uniformly converges to $f$.  
Then $\exists^{2}$ computes a Baire 1 representation for $f$ in terms of the double sequence. 
\end{lem}
\begin{proof}
That $f$ is Baire 1 follows from \cite{overderooie}*{Lemmas 11.7}.
The Baire 1 representation of $f$ is defined as follows:  since $(f_{n})_{n\in \N}$ converges uniformly to $f$, there is a sub-sequence $(g_{n})_{n\in \N}$ with 
\[\textstyle
(\forall n\in \N, x\in \R)(|f(x)-g_{n}(x)|<\frac{1}{2^{n}}  ).
\]
This subsequence can be computed using $\exists^{2}$ by restricting to $\Q$ and using the usual `$\eps/3$-trick'. 
Similarly, let $(G_{n,m})_{n,m\in \N}$ be the associated sub-sequence of $(F_{n,m})_{n,m\in \N}$.
Define $h_{n}=g_{n+1}-g_{n}$ and observe that $f=g_{1}+\sum_{n=1}^{\infty}h_{n}$ and $|h_{n}(x)|<\frac{1}{2^{n}}$ for all $x\in \R$ and $n\in \N$.
Clearly $H_{n, m}:=G_{n+1, m }-G_{n, m}$ converges to $h_{n}$ for $m\di \infty$.  
{If necessary}, modify $H_{n, m}$ to guarantee that, similar to $h_{n}$, we have $H_{n, m}(x)\leq \frac{1}{2^{n}}$ for all $m,n\in \N$ and $x\in \R$.  
Then $\lambda m.\lambda x.\sum_{n=1}^{\infty} H_{n, m}(x)$ is a sequence of continuous functions by Weierstrass' $M$-test.  
Moreover, this sequence converges to $\sum_{n=1}^{\infty}h_{n}$, as required, since $f=g_{1}+\sum_{n=1}^{\infty}h_{n}$.  
\qed
\end{proof}
Thirdly, we obtain the following version of the Baire characterisation theorem.
\begin{lem}\label{trux5}
Let $(f_{n})_{n\in \N}$ be a sequence of continuous functions with pointwise limit $f$ and $C$ a closed set.
Then $\exists^{2}$ computes a point of continuity $x\in C$ for $f_{\upharpoonright C}$ in terms of the sequence and an RM-code of $C$. 
\end{lem}
\begin{proof}
Given a Baire 1 representation of $f$, $\exists^{2}$ computes $\sup_{x\in [p, q]}f(x)$ as a sequence over $p, q\in \Q$ (\cite{dagsamXIV}*{\S2}).  
Hence, one readily obtains the oscillation function $\osc_{f}(x)$, defined as follows:
\[\textstyle
\textup{$\osc_{f}([a,b]):= \sup _{{x\in [a,b]}}f(x)-\inf _{{x\in [a,b]}}f(x)$ and $\osc_{f}(x):=\lim _{k \di \infty }\osc_{f}(B(x, \frac{1}{2^{k}}) ).$}
\]
We now have $D_{f}=\cup_{n\in \N}D_{n}$ where $D_{n}:= \{  x\in \R: \osc_{f}(x)\geq \frac{1}{2^{n}} \}$ is closed and nowhere dense. 
The Baire category theorem guarantees that $C_{f}=\cap_{n\in \N}C_{n}  \ne \emptyset$ where $C_{n}$ is the complement of $D_{n}$.
Moreover, given the Baire 1 representation of $f$, $\exists^{2}$ computes the point claimed to exist by the Baire category theorem (\cite{dagsamVII}*{\S6}).
The same works relative to any RM-closed $C$, and we are done.
\qed
\end{proof}
We believe that many results from \cite{overderooie} can be similarly made effective, yielding interesting results regarding $\Omega_{\bf F_{\sigma}}$. 
 
\subsection{Some equivalences}\label{clijsterf}
We identify some operations on ambivalent sets and functions that are exactly as hard as $\Omega_{\bf \Delta}$.  
To make this precise, we introduce the $\Omega_{\bf \Delta}$-cluster (Definition~\ref{specs}). 

\smallskip

First of all,  as in \cites{dagsamXIII, dagsamXV}, we now introduce the following equivalence class. 
\bdefi\label{specs}
We say that \emph{the functional $\Phi^{3}$ belongs to the $\Omega_{\bf \Delta}$-cluster} in case 
\begin{itemize}
\item the combination $\Phi+\exists^{2}$ computes $\Omega_{\bf \Delta}$, and 
\item the combination $\Omega_{\bf \Delta}+\exists^{2}$ computes $\Phi$. 
\end{itemize}
We say that \emph{the functionals $\Phi$ and $\Omega_{\bf \Delta}$ are computationally equivalent given $\exists^{2}$}.
\edefi
Secondly, to avoid complicated definitions and domain restrictions, we will sometimes abuse notation and make statements of the form
\begin{center}
\emph{any functional $\Psi$ satisfying a given specification $\textsf{\textup{(A)}}$ belongs to the $\Omega_{\bf \Delta}$-cluster.}
\end{center}
The centred statement means that for \textbf{any} functional $\Psi_{0}$ satisfying the given specification $\textsf{(A)}$, the combination $\Psi_{0}+\exists^{2}$ computes $\Omega_{\bf \Delta}$, while $\Omega_{\bf \Delta}+\exists^{2}$ computes \textbf{some} functional $\Psi_{1}$ satisfying the specification $\textsf{(A)}$.  

\smallskip

Thirdly, we now have the following `cluster theorem' for $\Omega_{\bf \Delta}$. 
We note that the third item generalises the Urysohn lemma while the fifth item is based on the pre-image characterisation of ambivalent functions. 
\begin{thm}\label{txno2}
The following are part of the $\Omega_{\bf \Delta}$-cluster.  
\begin{itemize}
\item Any functional $\Phi_{0}$ such that for $X\in \bf \Delta$, $\Phi_{0}(X)$ is an RM-code for $X\subset [0,1]$. 
\item Any functional $\Phi_{1}$ such that for $X\in \bf \Delta$, the sequence $\Phi_{1}(X)=(f_{n})_{n\in \N}$ is a Baire 1 representation of $\mathbb{1}_{X}$ \(\cite{overderooie}*{Lemma 11.6}\). 
\item \(Urysohn\) Any functional $\Phi_{1b}$ such that for disjoint $X, Y\in \bf \Delta$, $\Phi_{1b}(X, Y)$ is a Baire 1 function plus representation which is $0$ on $X$ and $1$ on $Y$ \(\cite{omaoma2}\).  
\item Any functional $\Phi_{2}$ such that for ambivalent and bounded $f:[0,1]\di \R$, the real $\Phi_{2}(f, p, q)$ equals $\sup_{x\in [p, q]}f(x)$ for any $p, q\in \Q$. 
\item Any functional $\Phi_{2b}$ such that for ambivalent $f:\R\di \R$ and open $V\in \R$, $\Phi_{2b}( f, V)$ equals an RM-code for the ambivalent set $f^{-1}(V)$.  
\item Any functional $\Phi_{3}$ such that for ambivalent and bounded $f:[0,1]\di \R$, the sequence $\Phi_{3}(f)$ equals a Baire 1 representation for $f$. 
\item Any functional $\Phi_{4}$ such that for $X\in \bf \Delta $, the real $\Phi_{4}(X)$ equals $\sup X$.
\item Any functional $\Phi_{5}$ such that for $X\in \bf \Delta $ and continuous $f:[0,1]\di \R$, the real $\Phi_{5}(X, f)$ equals $\sup_{x\in X}f(x)$.
\end{itemize}
\end{thm}
\begin{proof}
First of all, that $\Omega_{\bf \Delta}+\exists^{2}$ computes $\Phi_{0}$ is proved in Theorem \ref{5.11}.

\smallskip

Secondly, assume a functional $\Phi_{0}$ is given and fix $X\in \bf \Delta$.  Use the former to write $X= \cup_{n\in \N}F_{n}=\cap_{n\in \N}G_{n}$ where $F_{n}$ is open, $G_{n}$ is closed, and all are given by RM-codes.  
By Lemma \ref{trux2}, a Baire representation of $\mathbb{1}_{X}$ can be computed, and a functional $\Phi_{1}$ is thus computable from $\Phi_{0}+\exists^{2}$.
Now, $\exists^{2}$ computes the supremum of a Baire 1 function assuming a given Baire 1 representation by \cite{dagsamXIV}*{Theorem 2.6}.  
Hence, $\Phi_{1}$ is seen to compute $\Omega_{\bf \Delta}$ as $(\exists x\in \R)(x\in X)\asa (\exists n\in \N)[ \sup_{X\in [-n,n]} \mathbb{1}_{X}=1]$.  
Clearly, for $X\in \bf \Delta$, $X^{c}$ is disjoint from the former and also ambivalent.  Hence, $\Phi_{1b}$ computes $\Phi_{1}$, and vice versa, as required.  

\smallskip

Thirdly, to compute $\Phi_{2}$ from $\Omega_{\bf \Delta}$, let $(q_{n})_{n\in \N}$ be an enumeration of $\Q$, fix ambivalent $f:\R\di \R$, and use $\Omega_{\bf \Delta}$ to obtain sequences of reals $(a_{n, m,k})_{n,m,k\in \N}$ and $(b_{n, m,k})_{n,m,k\in \N}$ such that $\{ x \in [0,1] : f(x)<q_{k}\}=  \cap_{n\in \N}G_{n, k} $ with $G_{n, k}=\cup_{m\in \N}(a_{n,m,k}, b_{n,m,k})$.  
Clearly, we have
\begin{align}
&(\forall x\in [0,1])(f(x)<q_{k})\notag \\
&\asa  (\forall n\in \N)([0,1]\subset G_{n,k} )\notag\\
&\asa (\forall n\in \N)(\exists m_{0}\in \N)([0,1]\subset \cup_{m\leq m_{0}}(a_{n,m,k}, b_{n,m,k})),\label{kafff}
\end{align}
where the last step follows by (countable) Heine-Borel compactness.  The final formula \eqref{kafff} is (equivalent to) arithmetical, i.e.\ we can now find the supremum of $f$ using the usual interval-halving method (using $\exists^{2}$).   
To compute $\Omega_{\bf\Delta}$ from $\Phi_{2}$, one readily verifies that $\mathbb{1}_{X}$ is ambivalent for $X\in \bf \Delta$ (by a simple case distinction).  
The supremum as in the latter functional then readily yields the former functional.  For $\Phi_{2b}$, the latter is clearly computable from $\Phi_{0}$ while RM-codes for $f^{-}((q, +\infty))$ readily yield $\sup_{x\in [0,1]}f(x)$ using a formula similar to \eqref{kafff}; hence $\Phi_{2}$ follows, as required.  
 
\smallskip

Fourth, to compute $\Phi_{3}$ from $\Phi_{0}$, we proceed as follows.  Let $f:[0,1]\di \R$ be ambivalent with upper bound $m-1\in \N$ and lower bound $0$, and consider 
for $i<m$ the sets
\[\textstyle
A_{2i}^{m}:= \{ x\in [0,1]: i/m <f(x)< (i+1)/m  \}, A_{2i+1}^{m}:= \{x\in [0,1]: f(x)=i/m\}.
\] 
Each $A_{j}^{m}$ is ambivalent because $f$ is and these sets are pairwise disjoint.    
Now define $g_{m}:\R\di [0,1]$ as follows: $g_{m}(x):=\sum_{i=0}^{m}\mathbb{1}_{A_{i}^{m}}(x)\frac{i}{m}$.
By known results (\cite{overderooie}*{Lemma 11.6-7}), each function $g_{m}$ is Baire 1 and the uniform limit is a Baire 1 function, namely the original $f$.
We can thus use Lemmas \ref{trux2} and \ref{trux3} to obtain a Baire 1 representation of $f$, as required. 
To compute $\Phi_{2}$ form $\Phi_{3}$, we recall that $\exists^{2}$ computes the supremum of a Baire 1 function assuming a given Baire 1 representation by \cite{dagsamXIV}*{Theorem 2.6}.

\smallskip

To compute $\Phi_{4}$ from $\Phi_{0}$, fix $X\in \bf \Delta$ and let $\cup_{n\in \N}C_{n}=X$ where each $C_{n}$ is an RM-code for a closed set.  Then $\exists^{2}$ computes $\sup C_{n}$ and we put $\Phi_{4}(X):= \sup_{n\in \N}\sup C_{n}$.  
To compute $\Omega_{\bf \Delta}$ from $\Phi_{4}$, consider $X\in \bf \Delta$ and check if $0\in X$, if so, then $\Omega_{\bf \Delta}(X)=1$.  If not, then consider the ambivalent set $\{0\} \cup X$. Then $X$ is non-empty if and only if the supremum of the latter set is non-zero.  For $\Phi_{5}$, the latter computes $\Phi_{4}$ in case $f$ is the identity function.  To compute the former, proceed as for $\Phi_{4}$: fix $X\in \bf \Delta$ and let $\cup_{n\in \N}C_{n}=X$ where each $C_{n}$ is an RM-code for a closed set.  Then $\exists^{2}$ computes $\sup_{x\in C_{n}}f(x)$ and we put $\Phi_{5}(X, f):= \sup_{n\in \N}\sup_{x\in C_{n}}f(x)$.    
%
%
\qed
\end{proof}
In conclusion, various basic operations on ambivalent functions and sets are computationally equivalent.  
We believe there to be many more inhabitants of the $\Omega_{\bf \Delta}$-cluster.  

\subsection{Gauges and Baire 1 functions} 
 We introduce the equivalent `gauge' definition of Baire 1 functions (Definition~\ref{leepangpang}) and connect it to our structure functionals (Theorem \ref{txno}).
 
 \smallskip
 
First of all, we consider the following equivalent definition of Baire 1 function pioneered in \cite{leebaire}. We refer to the function $\delta$ as an \emph{$\eps$-gauge} of a Baire 1 function.  

\bdefi\label{leepangpang}
A function $f:\R\di \R$ is Baire 1 if for any $\eps>0$, there is $\delta:\R\di \R^{+}$ such that for all $x, y\in \R$, if $|x-y|<\min(\delta(x), \delta(y))$ then $|f(x)-f(y)|<\eps$.
\edefi
Secondly, obtaining gauges is easy from the computational viewpoint, as follows. 
\begin{thm}\label{trux4}
Let $(f_{n})_{n\in \N}$ be a sequence of continuous functions with pointwise limit $f$.
Then $\exists^{2}$ computes a gauge representation for $f$ in terms of the sequence. 
\end{thm}
\begin{proof}
Part of the proof of \cite{leebaire}*{Theorem 1} is effective, as follows. 
Let $(f_{n})_{n\in \N}$ be a sequence of continuous functions with pointwise limit $f$.  
Then $\exists^{2}$ readily computes a modulus of continuity $\delta_{n}(x, \eps)$, i.e.\ we have
\[
(\forall \eps>0, \forall n\in \N)(\forall x, y \in \R)( |x-y|<\delta_{n}(x, \eps)\di |f_{n}(x)-f_{n}(y)|<\eps ).  
\]
Similarly, let $\eta(x, \eps)$ be a modulus for the convergence of $f_{n}$ to $f$, i.e.\ 
\[
(\forall \eps>0, \forall x\in \R)(\forall n\geq \eta(x, \eps))(|f_{n}(x)-f(x)|<\eps). 
\]
Define $\delta_{\eps}(x):=  \max_{1\leq \eta(x, \eps)} \delta_{n}(x, \eps)$ and verify it is an $\eps$-gauge for $f$. \qed
\end{proof}
Thirdly, we connect the gauge definition to structure functionals in Theorem \ref{txno} below.  To this end, we need the following definition from \cite{dagsam} where the authors pioneered the computational study of Cousin's lemma (\cite{cousin1}).
\bdefi
Any $\Theta:(\R\di \R)\di \R$ is called a \emph{realiser for Cousin's lemma} if for any $\Psi[0,1]\di \R^{+}$, $\Theta(\Psi)=(x_{0}, \dots, x_{k})$ is such that $[0,1]\subset \cup_{i\leq k} B(x_{i}, \Psi(x_{i}))$. 
\edefi
Let $\Phi_{X}:((\R\di \R)\times \N) \di (\R\di \R)$ be such that $\Phi_{\Gamma}(f, k)$ is a $\frac{1}{2^{k}}$-gauge for $f$ in the class\footnote{We use `$\reg$' and `$\usco$' to denote the classes of regulated and usco functions.} $\Gamma$.  
The following theorem shows that computing $\eps$-gauges amounts to computing structure functionals, modulo a realiser for Cousin's lemma. 
\begin{thm}\label{txno}\smallskip
Let $\Theta$ be any realiser for the Cousin lemma. 
\begin{itemize}
\item The combination $\Phi_{\reg}+\Theta+\exists^{2}$ computes $\Omega_{b}$.  
\item The combination $\Phi_{\usco}+\Theta+\exists^{2}$ computes $\Omega_{C}$.  
\item The combination $\Phi_{\bf \Delta}+\Theta+\exists^{2}$ computes $\Omega_{\bf \Delta}$.  
\end{itemize}
\end{thm}
\begin{proof}
For the first part, let $X\subset \R$ be any set with at most one element.  
Then $\mathbb{1}_{X}$ is regulated and consider the $1/2$-gauge $\delta(x):=\Phi_{\reg}(f, 2)(x)$.  
If $x_{0}\in X$, then $x_{0}\not \in B(y, \delta(y))$ for any $y\in B(x_{0}, \delta(x_{0}))\setminus \{x_{0}\}$, by the definition of gauge function.  
Hence, if there is $x_{0}\in X$, then $\Theta(\lambda x.\delta(x), \overline{B}(x_{0}, \delta(x_{0})/2))$ must include the former real to provide a covering of $\overline{B}(x_{0}, \delta(x_{0})/2))$.  
In particular, we have 
\[
X\ne \emptyset\asa  (\exists p, q\in [0,1]\cap\Q)(\exists i\in \N)( \Theta(\lambda x.\delta(x), [p, q])(i)\in X  ).
\]
As a result, the combination $\Theta+\Phi_{\reg}+\exists^{2}$ computes $\Omega_{b}$.  
\smallskip

For the second part, let $C\subset [0,1]$ be closed and consider the usco function $\mathbb{1}_{C}$.  
Define the $1/2$-gauge $\delta(x):= \Phi_{\usco}(\mathbb{1}_{C}, 2) (x)$ and the set $O:= [0,1]\setminus C$.
If $x_{0}\in O$, then Feferman's $\mu$ suffices to find $N_{0}\in \N$ such that $(B(x_{0}, \frac{1}{2^{N_{0}}})\cap \Q)\subset O$.  
However, the ball $B(x_{0}, \frac{1}{2^{N_{0}}})$ can still contain points of $C$.  
We now show how to decide if there are such points (using the gauge and any $\Theta$-functional).  Suppose $x_{1}\not \in O$ but $x_{1}\in B(x_{0}, \frac{1}{2^{N_{0}}})$.    
Then $x_{1}\not \in B(y, \delta(y)) $ for any $y\in [O\cap B(x_{1}, \delta(x_{1}))]$, by the definition of gauge function.  
Hence, the finite sequence $\Theta(\lambda x.\delta(x), \overline{B}(x_{1}, \delta(x_{1})/2)  )$ must include $x_{2}\not \in O$ to provide a covering of $\overline{B}(x_{1}, \delta(x_{1})/2))$.  In this way, we obtain the equivalence between 
$(\exists x \in B(x_{0}, \frac{1}{2^{N_{0}}}))(x \not\in O )$ and 
\[\textstyle
  (\exists p, q\in \Q\cap [0,1])(\exists i\in \N)\big[ [p, q]\subset B(x_{0}, \frac{1}{2^{N_{0}}})\wedge \Theta(\lambda x.\delta(x), [p, q])(i)\not \in O  \big].  
\]
where the centred formula is decidable using $\exists^{2}$.
Hence, we can `shrink' $N_{0}$ (if necessary) to $N_{1}$ to guarantee $B(x_{0}, \frac{1}{2^{N_{1}}})\subset O$.
Restricting to the rationals in $O$, we obtain an RM-code for $O$, i.e.\ $\Omega_{C}$ is obtained.  

\smallskip

For the third part, fix $X\in \bf \Delta$ and consider $\mathbb{1}_{X}$.  The latter is Baire 1 by \cite{overderooie}*{Lemma~11.6} and ambivalent `by definition'.  
Now consider the $1/2$-gauge $\delta(x):= \Phi_{\bf \Delta}(\mathbb{1}_{X}, 2) (x)$. 
As in the previous, $(\exists x\in [0,1])(x \in X)$ is equivalent to 
\[\textstyle
  (\exists p, q\in \Q\cap [0,1])(\exists i\in \N)\big[ \Theta(\lambda x.\delta(x), [p, q])(i) \in X  \big],   
\]
which immediately yields $\Omega_{\bf \Delta}$.
\qed
\end{proof}
We observe that given a realiser for Cousin's lemma, $\Omega_{\Gamma}$ and $\Phi_{\Gamma}$ line up nicely.  

\smallskip

Finally, as explored in \cite{zulie}, gauges can be taken to be usco, but the construction is far from elementary, say compared to Theorem \ref{trux4}.

 \section{Recursion and selectors}\label{slef}
In this section, we show that the structure functional $\Omega_{\bf \Delta}$ can compute RM-codes for $\bf \Delta$-sets. 
As a corollary, the former also computes a \emph{selector} for $\bf \Delta$-sets, i.e.\ a functional that outputs an element
from a non-empty $\bf \Delta$-set.  
 
\smallskip
 
We start with some observations and lemmas.   We use `$P$' as a variable ranging over the sets in $\bf\Delta$ without always explicitly mentioning this.  
\begin{observation} 
Let $\Phi$ be a functional that outputs an RM-code for $P$ given a set $P \in \bf\Delta$ as input.  Then $\Omega_{\bf\Delta}$ is computable in $\Phi$ and $\exists^2$. 
\end{observation}
The following lemma goes back to Baire, and is proved by an application of the Baire category theorem for closed subsets of $\R$.
\begin{lem}\label{lemma.baire} 
Let $X$ be closed and nonempty and let $P\in \bf \Delta$. Then there is an open set $O$ such that $O \cap X \neq \emptyset$ and either $O \cap X \subset P$ or $O \cap X \subset P^c$. 
\end{lem}

We now let $(B_i)_{i \in \N}$ be a computable enumeration of all open intervals with rational endpoints, including the end-points $+\infty$ and $-\infty$. 
\bdefi\label{trandsfi}
By recursion on the ordinal $\alpha$, we define $A_\alpha \subseteq \N$ and the closed set $X_\alpha = \R \setminus \bigcup_{i \in A_\alpha}B_i$ as follows.
\begin{itemize}
\item We define $A_0 = \emptyset$ and for limit ordinals $\gamma$ we let $A_\gamma = \bigcup_{\alpha < \gamma} A_\alpha$.
\item We put $i \in A_{\alpha + 1}$ if $B_i \cap X_\alpha = \emptyset$ or if $B_i \cap X_\alpha\ne \emptyset$ and the latter set is contained either in $P$ or in $P^c$.
\end{itemize}
\edefi
Due to Lemma \ref{lemma.baire}, this transfinite recursion will terminate with empty $X_\alpha$ and, by construction, with $A_\alpha = \N$. Clearly each recursion step is computable in $P$, $P^c$ and $\Omega_{\bf\Delta}$. Further, the recursion will induce a pre-well-ordering (see Definition~\ref{preap}) of the final set $A_\alpha = \N$, and from this pre-well-ordering we can arithmetically define ${\bf F}_\sigma$-codes for $P$ and $P^c$, using $\Omega_{\bf\Delta} $ to tell us which new segments will go into $P$ and which will go into $P^c$.

\smallskip

It remains to prove that the aforementioned pre-well-ordering is computable in $\Omega_{\bf\Delta}$ and $\exists^2$. For this, we must introduce a few concepts.
\begin{defi}[Pre-well-orderings]\label{preap}~{
\begin{enumerate}
\item A \emph{pre-ordering} of a set $A \subseteq \N$ is a relation $\preceq$ on $A$ that is transitive and refexive. 
\item We write $a \simeq b$ if $a \preceq b$ and $b \preceq a$, and we write $a \prec b$ if $a \preceq b$ while it is not the case that $a \simeq b$.  If $a \prec b$ we will say that $a$ is below $b$. 
\item  A \emph{descending chain} is  a sequence $\{a_n\}_{n \in \N}$ such that $a_{n+1} \prec a_n$ for all $n$. 
\item A pre-ordering without descending chains is a \emph{pre-well-ordering}.
\end{enumerate}}
\end{defi}
All pre-orderings will have a maximal pre-well-ordered initial segment, consisting of those points below which there are no descending sequences. 

\smallskip

Now, we again consider the inductive definition as in Definition \ref{trandsfi}.  
\begin{defi}
A pre-ordering $(\N, \preceq)$ is \emph{agreeable} if the following is satisfied.
\begin{center}
 For each $a \in \N$, let $X_a = \R \setminus \bigcup_{b \prec a} B_b$. 
 Then either $B_a \cap X_a = \emptyset$ or $B_a \cap X_a \neq \emptyset$ and $B_a \cap X_a$ is either contained in $P$ or in $P^c$.
\end{center}
\end{defi}
This reflects the recursion step in the recursion described above.

\smallskip

By transfinite induction, we also see that given an agreeable pre-ordering, the pre-well-ordered initial segment will coincide with an initial segment of the one obtained from the recursion.  Finally, we see that $\exists^2$ and $\Omega_{\bf\Delta}$ can decide if a pair $(A, \preceq)$ is an agreeable pre-ordering in our sense.

\begin{lem}\label{lemma.one} 
If $(\N , \preceq)$ is an agreeable pre-ordering, it is actually the one and only agreeable pre-well-ordering on $\N$.
\end{lem}
\begin{proof} Assume not, and let $A^w$ be the domain of the pre-well-ordered initial segment. Let $X^w = \R \setminus \bigcup_{a \in A^w}B_a$. We must have that $X^w \neq \emptyset$ since otherwise the induction would have stopped at each point in the non-pre-well-ordered part. Choose $n$ such that $B_n \cap X^w \neq \emptyset$ and is contained in either $P$ or $P^c$. Now, let $c \in \N \setminus A^w$. Then either $B_n \cap X_c = \emptyset$ or $n \simeq c$ (since they both will satisfy the requirements of the recursion step). Since $c \in \N \setminus A^w$ was arbitrary, it follows that $B_n \cap X_c = \emptyset$ for all $c \in \N \setminus A^w$, contradicting the choice of $n$.
\qed
\end{proof}
The following result is the non-trivial direction in Theorem \ref{txno2}.  Here $\Omega_1$ is the functional selecting the one element in a singleton set.  
The functional $\Omega_1$ is computable in all the structure functionals $\Omega_{\Gamma}$ we have considered so far.
\begin{thm}\label{5.11} 
There is a functional $\Phi$ which is computable in $\Omega_{\bf\Delta}+\exists^2$ and such that $\Phi(X)$ is an RM-code for any $X\in \bf \Delta$.
\end{thm}
\begin{proof}
The class of agreeable pre-orderings is computable in $\Omega_{\bf\Delta}$ and $\exists^2$. By Lemma \ref{lemma.one} there is exactly one agreeable pre-ordering, and that is the one obtained by the transfinite recursion. If we apply $\Omega_1$ to this class we obtain this pre(well)ordering, and from that one we read off the ${\bf F}_\sigma$-codes for $P$ and $P^c$.
\qed
\end{proof}
\begin{cor} We have the following. 
\begin{enumerate}
\renewcommand{\theenumi}{\alph{enumi}}
\item There is a selector operator for $\bf \Delta$-sets computable in $\Omega_{\bf\Delta}$.\label{za}
\item The functional $\Omega_{\rm countable}$,  and thus $\Omega_{{\bf F}_\sigma}$, is not computable in $\Omega_{\bf\Delta} +\exists^2$.\label{zb}
\end{enumerate}
\end{cor}
\begin{proof} 
Given $P\in \bf \Delta$, $\Omega_{\bf\Delta}$ provides an ${\bf F}_\sigma$-code for an increasing  sequence $(P_k)_{k \in \N}$ of closed sets with $P$ as the union. When $P$ is non-empty, we extract the least element in the first non-empty $P_k\cap [-k,k]$. This proves item \ref{za} from the theorem. Item \ref{zb} also follows because whenever a functional $\Omega_\Gamma$ provides a selector for $\Gamma$, then $\Omega_\Gamma$ cannot compute $\Omega_{\rm countable}$. 
\qed
\end{proof}

\begin{credits}
\subsubsection*{\ackname} 
The research of the second author was
supported by the Klaus Tschira Boost Fund via the grant Projekt KT 43. We express our gratitude
towards the latter institution.

\subsubsection*{\discintname}
To the best of our knowledge, we have no competing interests to declare that are
relevant to the content of this article. 
\end{credits}

\section*{References}

 \end{document}